\documentclass{article}
\usepackage[a4paper,scale=0.8]{geometry}
\usepackage{amsmath}
\usepackage{amsthm}
\usepackage{amssymb}
\usepackage{bm}  
\usepackage{url}
\usepackage[numbers]{natbib}
\usepackage{enumerate}
\usepackage{graphicx}
\usepackage{blkarray}
\usepackage{color}
\usepackage{caption}
\usepackage{subcaption}

\newcommand{\conv}{\mathop{\mathrm{conv}}}
\newcommand{\nnegrk}{\mathop{\mathrm{rank}_+}} 
\newcommand{\suppmat}{\mathop{\mathrm{suppmat}}} 

\newcommand{\xc}{\mathop{\mathrm{xc}}} 

\newcommand{\vct}{\bm}

\newcommand{\K}{{\mathrm{K}}}
\newcommand{\N}{{\mathrm{N}}}

\newcommand{\RR}{{\mathbb{R}}}

\newcommand{\CutP}{\mathrm{CUT}^\square}

\newcommand{\CorP}{\mathrm{COR}^\square}
\newcommand{\TSP}{{\mathrm{TSP}}}
\newcommand{\SAT}{{\mathrm{SAT}}}
\newcommand{\STAB}{{\mathrm{STAB}}}
\newcommand{\threeDM}{{\mathrm{3DM}}}
\newcommand{\SUBSETSUM}{{\mathrm{SUBSETSUM}}}
\newcommand{\Knapsack}{{\mathrm{KNAPSACK}}}

\newcommand*{\abs}[1]{\lvert#1\rvert}

\title{On the extension complexity of combinatorial polytopes }
\author{David Avis$^{1,2}$ 
\thanks{Email: \texttt{avis@cs.mcgill.ca}}
\and Hans Raj Tiwary$^3$
  \thanks{Email: \texttt{hans.raj.tiwary@ulb.ac.be}}}

\theoremstyle{plain}
\newtheorem{theorem}{Theorem}
\newtheorem{corollary}{Corollary}
\newtheorem{lemma}{Lemma}
\newtheorem{proposition}{Proposition}

\theoremstyle{definition}

\theoremstyle{remark}

\let\cite\citep

\BAnewcolumntype{s}{>{$\small$}c}

\newlength{\normalparindent}
\newlength{\normalparskip}
\begin{document}
\maketitle
\footnotetext[1]{GERAD and School of Computer Science, McGill University,
  3480 University Street, Montreal, Quebec, Canada H3A 2A7.}
\footnotetext[2]{Graduate School of Informatics,
  Kyoto University, Sakyo-ku, Yoshida Yoshida, Kyoto 606-8501, Japan}
\footnotetext[3]{Department of Mathematics, Universit\'{e} Libre de Bruxelles, Boulevard du Triomphe, B-1050 Brussels,  Belgium}

\begin{abstract}
In this paper we extend recent results of Fiorini et al. on the extension complexity
of the cut polytope and related polyhedra.
We first describe a lifting argument to show exponential extension
complexity for a number of NP-complete problems including subset-sum and three
dimensional matching.
We then obtain a relationship between the extension complexity
of the cut polytope of a graph and that of its graph minors. Using this we
are able to show exponential extension complexity for the cut polytope of
a large number of graphs,
including those used in quantum information and suspensions of cubic
planar graphs.
\end{abstract}

\section{Introduction}
\label{intro}
In formulating optimization problems as linear programs (LP), adding extra variables can greatly reduce the size of the LP \cite{ConfortiCornuejolsZambelli10}. However, it has been shown recently that for some polytopes one cannot obtain polynomial size LPs by adding extra variables \cite{FMPTW, Rothvoss11}. In a recent paper \cite{FMPTW}, Fiorini et.al. proved such results for the cut polytope, the traveling salesman polytope, and the stable set polytope for the complete graph $K_n.$ In this paper, we extend the results of Fiorini et. al. to several other interesting polytopes. We do not claim novelty of our techniques, in that they have been used - in particular - by Fiorini et. al. Our motivation arises from the fact that there is a strong indication that NP-hard problems require superpolynomial sized linear programs. We make a step in this direction by giving a simple technique that can be used to translate NP-completeness reductions into  lower bounds for a number of interesting polytopes.

\paragraph*{Cut polytope and related polytopes.}~
The cut polytope arises in many application areas and has been extensively studied. 
Formal definitions of this polytope and its relatives are given in the next section.
A comprehensive compilation of facts about the cut polytope is contained in
the book by Deza and Laurent
\cite{cutbook}. Optimization over the cut polytope is known as the max cut problem,
and was included in Karp's original list of problems that he proved to be NP-hard.
For the complete graph with $n$ nodes,
a complete list of the facets of the cut polytope
$\CutP_n$ is known for
$n\le7$~ (see Section 30.6 of \cite{cutbook}),
as well as many classes of facet producing valid inequalities.
The hypermetric inequalities (see Chapter~28 of \cite{cutbook}) 
are examples of such a class, and it is known that an exponential
number of them are facet inducing.
Less is known about classes of facets for the cut polytope of
an arbitrary graph, $\CutP(G)$. Interest in such polytopes arises
because of their application to fundamental problems in physics. 

In quantum information theory, the cut polytope arises in relation to Bell inequalities.
These inequalities, a generalization of Bell's original
inequality~\cite{Be64},
were introduced
to better understand the nonlocality of quantum physics.
Bell inequalities for two parties are inequalities valid for the
cut polytope of the complete tripartite graph $\K_{1,n,n}$.
Avis, Imai, Ito and Sasaki~\cite{AIIS}
proposed an operation named \emph{triangular elimination}, which is a
combination of zero-lifting and Fourier-Motzkin elimination (see e.g.\
\cite{Zi95}) using the triangle inequality.
They proved that triangular elimination maps facet inducing inequalities of the
cut polytope of the complete graph to facet inducing inequalities of
the cut polytope of $\K_{1,n,n}$.
Therefore a standard description of such polyhedra contains an exponential number
of facets.

In \cite{AII} the method was extended to obtain facets of $\CutP(G)$
for an arbitrary graph $G$ from facets
of $\CutP_n$. For most, but not all classes of graphs, $\CutP(G)$ has an exponential number of
facets. An interesting exception are the graphs with no $K_5$ minor. 
Results of Seymour for the cut cone, extended by Barahona
and Mahjoub to the cut polytope (see Section 27.3.2 of \cite{cutbook}), show that
the facets in this case
are just projections of triangle inequalities.
It follows that the max cut problem for a graph $G$ on $n$ vertices
with no $K_5$ minor can be solved in polynomial time
by optimizing over the {\em semi-metric polytope}, which has $O(n^3)$ facets.
Another way of expressing this is to say that in this
case $\CutP(G)$ has
$O(n^3)$ \emph{extension complexity}, a notion that will be discussed next.

\paragraph*{Extended formulations and extensions}
~
Even for polynomially solvable problems, the associated polytope 
may have an exponential number of facets. By working in a higher dimensional space
it is often possible to decrease the number of constraints. 
In some cases, a polynomial increase in dimension can yield
an exponential decrease in the number of constraints. 
The previous paragraph contained an example of this.

For NP-hard problems the notion of extended formulations also comes into play.
Even though a natural LP formulation of such a problem has exponential size,
this does not rule out a polynomial size formulation in higher dimensions.

In a groundbreaking paper, Yannakakis \cite{Ya91} proved 
that every symmetric LP for the 
Travelling Salesman Problem (TSP) has exponential size.
Here, an LP is called \emph{symmetric} if every permutation of the 
cities can be extended to a permutation of all the variables of the 
LP that preserves the constraints of the LP. 
This result refuted various claimed proofs of a polynomial time algorithm for the
TSP. In 2012 Fiorini et al. \cite{FMPTW} proved that the max cut problem also
requires exponential size if it is to be solved as an LP. Using this result,
they were able to drop the symmetric condition, 
required by Yannakakis, to get a general
super polynomial bound for LP formulations of the TSP.

\paragraph*{Our contributions and outline of the paper}~
In this paper, we provide more examples of some polytopes associated with hard combinatorial 
problems as a way to illustrate a general technique for proving lower bounds for the extension complexity of a polytope.
The rest of the paper is organized as follows. 

In the next 
section we give background on cut polytopes, a summary of the approach in \cite{Ya91} and \cite{FMPTW}, 
and discuss a general strategy 
for proving lower bounds. In Section \ref{sec:examples} we discuss four polytopes arising 
from the 3SAT, subset sum, 3-dimensional matching and the maximum stable set problems, and prove 
superpolynomial extension complexity for them. For the stable set polytope, we improve the 
result of \cite{FMPTW} by proving superpolynomial lower bounds for the stable set polytope of cubic planar graphs.

In Section \ref{sec:ext_cut} we first reprove the result of \cite{FMPTW} for 
the cut polytope directly without making use of the correlation polytope. 
We then prove how the bounds propagate when one takes the minors of a graph. 
We use our results to prove superpolynomial lower bounds for the Bell-inequality polytope $\CutP(K_{1,n,n})$ 
described above.
As already noted, the max cut problem can be solved in polynomial time for graphs that are $K_5$ minor free
and their cut polytope has a polynomial size extended formulation. Planar graphs are a subset of this class. 
A suspension of a graph is formed by adding an additional vertex and joining it to all of the graph's original vertices.
Barahona \cite{Barahona83} proved that the max cut problem is NP-hard for 
suspensions of planar graphs and hence for $K_6$ minor-free graphs. 
We show that this class of graphs has superpolynomial extension complexity. 
In fact, the graphs used in our proof are suspensions of cubic planar graphs. 

\section{Preliminaries} \label{sect:prelim}

We briefly review basic notions about the cut polytope and extension complexity
used in later
sections.
Definitions, theorems and other results for the cut polytope stated in this section are
from \cite{cutbook},
which readers are referred to for more information.
We assume that readers are familiar with basic notions in convex
polytope theory such as convex polytope, facet, projection and
Fourier-Motzkin elimination.
Readers are referred to a textbook~\cite{Zi95} for details.

Throughout this paper, we use the following notation.
For a graph $G=(V,E)$ we denote the edge between two vertices $u$ and $v$ by $uv$,
and the neighbourhood
of a vertex $v$ by $\N_G(v)$. We let $[n]$ denote the integers $\{ 1,2,...,n\}$.

\subsection{Cut polytope and its relatives}
\label{cutdef}

  The \emph{cut polytope} of a graph $G=(V,E)$, denoted $\CutP(G)$, is
  the convex hull of the cut vectors $\vct{\delta}_G(S)$ of $G$
  defined by all the subsets $S\subseteq V$ in the
  $\abs{E}$-dimensional vector space $\RR^E$.
  The cut vector $\vct{\delta}_G(S)$ of $G$ defined by $S\subseteq V$
  is a vector in $\RR^E$ whose $uv$-coordinate is defined as follows:
  \[
    \delta_{uv}(S)=\begin{cases}
      1 & \text{if $\abs{S\cap\{u,v\}}=1$,} \\
      0 & \text{otherwise,}
    \end{cases} \text{ for $uv\in E$.}
  \]
  If $G$ is the complete graph $\K_n$, we simply denote $\CutP(\K_n)$ 
  by $\CutP_n$.

For completeness, although we will not use it explicitly, we define the \emph{correlation polytope}
$\CorP_n$. For each subset $S\subseteq \{1,2,...,n\}$ we define the correlation vector
$\pi(S)$ of length $(n+1)n/2$ by setting $\pi(S)_{ij} = 1$ if and
only if $i,j \in S$, for all $1 \le i \le j \le n$. $\CorP_n$ is the convex hull of the $2^n$ correlation vectors $\pi(S)$.
A linear map, known as the covariance map, shows the one-to-one correspondence of
$\CorP_n$ and $\CutP_{n+1}$ (see \cite{cutbook}, Ch. 5).

For a subset $F$ of a set $E$, the \emph{incidence vector}
of $F$ (in $E$)%
\footnote{The set $E$ is sometimes not specified explicitly
  when $E$ is clear from the context or the choice of $E$ does not
  make any difference.}
is the vector $\vct{x}\in\{0,1\}^E$ defined by
$x_e=1$ for $e\in F$ and $x_e=0$ for $e\in E\setminus F$.
Using this term, the definition of the cut vector can also be stated
as follows: $\vct{\delta}_G(S)$ is the incidence vector of the cut set
$\{uv\in E\mid\abs{S\cap\{u,v\}}=1\}$ in $E$.
When $G=K_n$ we simply denote the cut-vectors by $\delta(S)$.

We now describe an important 
well known general class of valid inequalities for $\CutP_n$ (see, e.g. \cite{cutbook}, Ch. 28).

\begin{lemma} \label{lem:hyp}
For any $n \ge 2$, let $b_1, b_2, ... , b_n$ be any set of $n$ integers. 
The following inequality
is valid for $\CutP_n$:
\begin{equation} \label{eq:hyp}
\sum_{1 \leq i < j \leq n} b_i b_j x_{ij} \leq\left \lfloor 
\frac{(\sum_{i=1}^n b_i )^2}{4} \right\rfloor
\end{equation}
\end{lemma}
\begin{proof}
Let $\delta(S)$ be any cut vector for the complete graph $K_n$. Then
\begin{equation} \label{eq:hyp1}
\sum_{1 \leq i < j \leq n} b_i b_j \delta(S)_{ij}  =
(\sum_{i \in S} b_i)(\sum_{i \notin S} b_i )
\end{equation}
Now observe that if the sum of the $b_i$ is even the floor sign
is redundant and an elementary calculation shows that the right hand
side of (\ref{eq:hyp1}) is bounded above by the right hand side of
(\ref{eq:hyp}). If the sum of the $b_i$
are odd then the same calculation gives an upper bound of 
$(\sum_{i=1}^n b_i+1)(\sum_{i=1}^n b_i-1)/4= (\sum_{i=1}^n b_i)^2/4 - 1/4$
on the right hand
side of (\ref{eq:hyp1}) and the lemma follows.
\end{proof}

The inequality (\ref{eq:hyp}) is called {\em hypermetric} (respectively, of {\em negative type})
if the integers $b_i$
can be partitioned into two subsets whose sum differs by one (respectively, zero).
A simple example of hypermetric inequalities are the triangle inequalities, obtained
by setting three of the $b_i$ to be +/- 1 and the others to be zero.
The most basic negative type inequality is non-negativity, obtained by setting one
$b_i$ to 1, another one to -1, and the others to zero.
We note in passing that Deza (see Section 6.1 of \cite{cutbook})
showed that each negative type inequality
could be written as a convex combination of hypermetric inequalities, so that none
of them are facet inducing for $\CutP_n$.

For any fixed $n$ there are an infinite number of hypermetric inequalities, but all but a finite number
are redundant. This non-trivial fact was proved by Deza, Grishukhin and Laurent (see \cite{cutbook} Section 14.2)
and allows us to define the {\em hypermetric polytope}, which we will refer to again later.

\subsection{Extended formulations and extensions}

In this paper we make use of the machinery developed and described in
Fiorini et al. \cite{FMPTW}. A brief summary is given here and
the reader is referred to the original paper for more details and proofs.

An \emph{extended formulation} (EF) of a polytope $P \subseteq \RR^d$ 
is a linear system

\begin{equation} \label{eq:EF}
E x + F y = g,\ y \geqslant \mathbf{0}
\end{equation}
in variables $(x,y) \in \RR^{d+r},$ where $E, F$ are real 
matrices with $d, r$ columns respectively, and $g$ is a column vector, 
such that $x \in P$ if and only if there exists $y$ such 
that \eqref{eq:EF} holds. The \emph{size} of an EF is defined as its 
number of \emph{inequalities} in the system.

An \emph{extension} of the polytope $P$ is another polytope
$Q \subseteq \mathbb{R}^e$ such that $P$ is the image of $Q$ under a linear map. 
Define the \emph{size} of an extension $Q$ as the number of facets of $Q$. Furthermore, define the \emph{extension complexity} of $P$, denoted by $\xc{( P )},$ as the minimum size of any extension of $P.$ 

For a matrix $A$, let $A_i$ denote the $i$th row of 
$A$ and $A^j$ to denote the $j$th column of $A$.
Let $P = \{x \in \RR^d \mid Ax \leqslant b\} = \conv(V)$ be a polytope, 
with $A \in \RR^{m \times d}$, $b \in \RR^m$ and 
$V = \{v_1,\ldots,v_n\} \subseteq \RR^d$. 
Then \(M \in \RR_+^{m \times n}\) defined as 
\(M_{ij} := b_i - A_i v_j\) with \(i \in [m] := \{1,\ldots,m\}\) 
and \(j \in [n] := \{1,\ldots,n\}\) is 
the \emph{slack matrix} of \(P\) w.r.t.\ $Ax \leqslant b$ and $V$. 
We call the submatrix of $M$ induced by rows 
corresponding to facets and columns corresponding to 
vertices the \emph{minimal slack matrix} of $P$ and
denote it by~$M(P)$. Note that the slack matrix
may contain columns that correspond to feasible points that are not vertices of $P$ and
rows that correspond to valid inequalities that are not facets of $P$, and therefore the slack matrix of a polytope is not a uniquely defined object. However every slack matrix of $P$ must contain rows and columns corresponding to facet-defining inequalities and vertices, respectively. 
As observed in \cite{FMPTW}, for proving bounds on the extension complexity of a polytope $P$ it suffices to
take any slack matrix of $P$.
Throughout the paper we refer to the minimal slack matrix of $P$ as \emph{the} slack matrix of $P$ and 
any other slack matrix as \emph{a} slack matrix of $P.$

A \emph{rank-$r$ nonnegative factorization} of a (nonnegative) 
matrix $M$ is a factorization $M = QR$ where $Q$ and $R$ are nonnegative 
matrices with $r$ columns (in case of \(Q\)) and $r$ rows 
(in case of \(R\)), respectively. The nonnegative rank of $M$ 
(denoted by: $\nnegrk(M)$) is thus simply the minimum rank of a 
nonnegative factorization of $M$. Note that $\nnegrk(M)$ is also 
the minimum $r$ such that $M$ is the sum of $r$ nonnegative rank-$1$ matrices. 
In particular, the nonnegative rank of a matrix $M$ is at least 
the nonnegative rank of any submatrix of $M$.

The following theorem shows the equivalence of nonnegative rank
of the slack matrix, extension and
size of an EF.

\begin{theorem}[Yannakakis \cite{Ya91}] \label{thm:factorization}
Let $P=\{x\in\RR^d\mid Ax\leqslant b\} = \conv (V)$ be a polytope 
with $\dim(P) \geqslant 1$ with a slack matrix $M$. 
Then the following are equivalent for all positive integers $r$:

\begin{enumerate}[(i)]
\item $M$ has nonnegative rank at most $r$;
\item $P$ has an extension of size at most $r$ (that is, with at most $r$ facets);
\item $P$ has an EF of size at most $r$ (that is, with at most $r$ inequalities).
\end{enumerate}
\end{theorem}

For a given matrix $M$
let $\suppmat(M)$ be the binary support matrix of $M$, so
$$
\suppmat(M)_{ab} = \left\{
\begin{array}{ll}
1 &\text{if } M_{ab}\neq 0,\\
0 &\text{otherwise}.
\end{array}
\right.
$$
A {\em rectangle} is the cartesian product of a 
set of row indices and a set of column indices. 
The {\em rectangle covering bound} is the minimum number of
monochromatic rectangles are needed to cover all the 1-entries of the 
support matrix of $M$.
In general it is difficult to calculate the nonnegative rank of a matrix
but sometimes a lower bound can be obtained as shown in the next theorem.

\begin{theorem}[Yannakakis \cite{Ya91}] \label{thm:nnegrkvsndetcc}
Let $M$ be any matrix with nonnegative real entries and 
$\suppmat(M)$ its support matrix. 
Then $\nnegrk(M)$ is lower bounded by the rectangle covering 
bound for $\suppmat(M)$.
\end{theorem}

The following $2^n \times 2^n$ matrix $M^* = M^*(n)$ with rows and 
columns indexed by $n$-bit strings $a$ and $b$, and real nonnegative entries
$$
M^*_{ab} := (a^{\intercal} b -1)^2.
$$
is very useful for obtaining exponential bounds on the EF of various polytopes.
This follows from the following result.

\begin{theorem}[De Wolf \cite{Wo03}]\label{thm:coverlowerboundforM}
Every 1-monochromatic rectangle cover of $\suppmat(M^*(n))$ has size $2^{\Omega(n)}$.
\end{theorem}

\begin{corollary} $\nnegrk{(M^*(n))} \geqslant 2^{\Omega(n)}.$\end{corollary}

Using these ingredients, Fiorini et al. \cite{FMPTW} proved the following fundamental result, 

\begin{theorem}[Lower Bound Theorem]\label{LBT}
Let $M(n)$ denote the slack matrix, of $\CutP_n$, extended with a suitably chosen set of $2^n$ redundant inequalities. 
Then $M^*(n-1)$ occurs as a submatrix of $M(n)$ and hence $\CutP_n$ has extension complexity $2^{\Omega(n)}.$
\end{theorem}

They further proved a $2^{\Omega(\sqrt n)}$ lower bound on the size of extended formulations for 
the travelling salesman polytope, $\TSP(n)$, by embedding $\CutP_n$ as a face of $\TSP(m)$ where $m=O(n^2).$ 
A similar embedding argument was used to show the same lower bound applies to the stable set polytope,
$\STAB(n)$.


\subsection{Proving lower bounds for extension complexity}\label{subsec:proving_lower_bound}
\label{method}
Suppose one wants to prove a lower bound on the extension complexity for a polytope $P$. 
Theorem \ref{LBT} provides a way to do it from scratch: construct a non-negative matrix that has a high non-negative rank and 
then show that this matrix occurs as a submatrix of a slack matrix of $P.$ 
Clearly this can be very tricky since there exists neither a general framework for creating such a matrix for each polytope, 
nor a general way of using a result for one class of polytopes for another.

We now note two observations that are useful in translating results from one polytope to another. Let $P$ and $Q$ be two polytopes. Then,

\begin{proposition} \label{prop:ef} If $P$ is a projection of $Q$ then $\xc{(P)}\leqslant\xc(Q).$\end{proposition}

\begin{proposition} \label{prop:face}If $P$ is a face of $Q$ then $\xc{(P)}\leqslant\xc(Q).$\end{proposition}

Naturally there are many other cases where the conditions of neither of these propositions apply and yet a lower bounding argument for one polytope can be derived from another.
However we would like to point out that these two propositions already seem to be very powerful. 
In fact, out of the three lower bounds proved by Fiorini et. al. \cite{FMPTW} two (for $\TSP(n)$ and $\STAB(n)$) use these propositions, 
while the lower bound on the cut polytope is obtained by showing a direct embedding of $M^*(n)$ in the slack matrix of $\CutP_n.$

Fiorini et. al. \cite{FMPTW} first show $M^*(n)$ is a submatrix of the slack matrix of the correlation polytope $\CorP_n$ 
and then use its affine equivalence with $\CutP_{n+1}.$ This is followed by an embedding of $\CutP_n$ 
as a face of $\STAB(G(n))$ where $G(n^2)$ is a graph with $O(n^2)$ vertices and $O(n^2)$ edges 
implying a \emph{worst case} lower bound of $2^{\Omega(\sqrt{n})}$ for the extension complexity 
of the stable set polytope of a graph with $n$ vertices. Similarly, worst case 
lower bounds are obtained for the traveling salesman polytope by embedding $\CorP_n$ in a face of $\TSP(n^2).$

In the next section we will use these propositions to show superpolynomial lower bounds on the extension complexities of polytopes associated with four NP-hard problems.

\section{Polytopes for some NP-hard problems}\label{sec:examples}

In this section we use the method of Section \ref{method} to show super polynomial extension complexity for polytopes
related to the following problems: subset sum, 3-dimensional matching and stable set for cubic planar graphs.
These proofs are derived by applying this method to standard reductions from 3SAT, which is our starting point.

\subsection{3SAT}\label{subsec:3sat}
For any given 3SAT formula $\Phi$ with $n$ variables in conjunctive normal form define the polytope $\SAT(\Phi)$ as the convex hull of all satisfying assignments. That is, $$\SAT(\Phi):= \conv(\{x\in[0,1]^n\mid \Phi(x)=1 \})$$

The following theorem and its proof are implicit in \cite{FMPTW}, who make use of the correlation polytope.
We provide the proof for completeness, stated this time in terms of the cut polytope. 

\begin{theorem} For every $n$ there exists a 3SAT formula $\Phi$ with $O(n)$ variables and $O(n)$ clauses such that $\xc(\SAT(\Phi))\geqslant 2^{\Omega(\sqrt n)}$. 
\end{theorem}
\label{3dmthm}
\begin{proof}
For the complete graph $K_m$ we define a boolean formula $\Phi_m$ in conjunctive normal form over the variables $x_{ij}$ for $i,j\in\{1,\ldots,m\}$ such that every clause in $\Phi_m$ has three literals and $\CutP(K_m)$ is a projection of $\SAT(\Phi_m).$

Consider the relation $x_{ij}=x_{ii}\oplus x_{jj},$ where $\oplus$ is the xor operator. The boolean formula $$(x_{ii}\lor\overline{x}_{jj}\lor x_{ij})\land(\overline{x}_{ii}\lor{x_{jj}}\lor {x_{ij}})\land({x_{ii}}\lor{x_{jj}}\lor \overline{x}_{ij})\land(\overline{x}_{ii}\lor\overline{x}_{jj}\lor \overline{x}_{ij})$$ is true if and only if $x_{ij}=x_{ii}\oplus x_{jj}$ for any assignment of the variables $x_{ii},x_{jj}$ and $x_{ij}.$

Now define $\Phi_m$ as 
$$\Phi_m := \bigwedge_{i,j \in [m] \atop i \neq j} \left[ (x_{ii}\lor\overline{x}_{jj}\lor x_{ij})\land(\overline{x}_{ii}\lor{x_{jj}}\lor {x_{ij}})\land({x_{ii}}\lor{x_{jj}}\lor \overline{x}_{ij})\land(\overline{x}_{ii}\lor\overline{x}_{jj}\lor \overline{x}_{ij}) \right].
$$

It is easy to see that any vertex of $\SAT(\Phi_m)$ can be projected to a vertex of $\CutP(K_m)$ by projecting out the variables $x_{ii}$ for $i\in\{1,\ldots,m\}$ since $x_{ij}=1$ if and only if $x_{ii}$ and $x_{jj}$ are assigned different values, and hence the assignment defines a cut in $K_m.$ Furthermore, any vertex of $\CutP(K_m)$ can be extended to any of the two assignments that correspond to the cut defined by the vector. That is, if a cut vector of $\CutP(K_m)$ partitions the set of vertices into $S$ and $\overline{S}$ then extending the cut vector by assigning $x_{ii}=1$ if $i\in S$ and $x_{ii}=0$ if $i\in\overline{S}$ (or the other way round) defines a satisfying assignment for $\Phi_m$ and therefore a vertex of $\SAT(\Phi_m).$

Therefore, $\CutP(K_m)$ is a projection of $\SAT(\Phi_m),$ and by Proposition 1 we can conclude that $\xc(\SAT(\Phi_m))\geqslant \xc(\CutP(K_m)) \geqslant 2^{\Omega(m)}.$ Note that $\Phi_m$ has $O(n^2)$ variables and clauses. Therefore, we have the desired result.
\end{proof}

\subsection{Subset sum}
The subset sum problem is a special case of the knapsack problem. Given a set of $n$ integers $A=\{a_1,\ldots,a_n\}$ and another integer $b,$ the subset sum problems asks whether any subset of $A$ sums exactly to $b.$ Define the subset sum polytope $\SUBSETSUM(A,b)$ as the convex hull of all characteristic vectors of the subsets of $A$ whose sum is exactly $b.$
$$\SUBSETSUM(A,b):=\text{conv}\left(\left\{x\in[0,1]^n\mid \sum_{i=1}^n a_i x_i=b\right\}\right)$$

The subset sum problem then is asking whether $\SUBSETSUM(A,b)$ is empty for a given set $A$ and 
integer $b.$ Note that this polytope is a face of the knapsack polytope 
$$\Knapsack(A,b):=\text{conv}\left(\left\{x\in[0,1]^n\mid \sum_{i=1}^n a_i x_i\leqslant b\right\}\right)$$

In this subsection we prove that the subset sum polytope (and hence the knapsack polytope) can have superpolynomial extension complexity.

\begin{theorem} For every 3SAT formula $\Phi$ with $n$ variables and $m$ clauses, there exists a set of integers $A(\Phi)$ and integer $b$ with $|A|=2n+2m$ such that $\SAT(\Phi)$ is the projection of $\SUBSETSUM(A,b).$
\end{theorem}
\label{ssthm}
\begin{proof}
Suppose formula $\Phi$ is defined in terms of variables $x_1, x_2, ... , x_n$ and clauses $C_1, C_2, ..., C_m$.
We use a standard reduction from 3SAT to subset sum (e.g., \cite{corman}, Section 34.5.5). We define $A(\Phi)$ and $b$ as follows. Every integer in $A(\Phi)$ as well as $b$ is an $(n+m)$-digit number (in base 10). The first $n$ bits correspond to the variables and the last $m$ bits correspond to each of the  clauses.
$$ b_j=\begin{cases} 1, & \mbox{if } 1\leqslant j\leqslant n \\4, & \mbox{if } n+1\leqslant j\leqslant n+m \end{cases}.$$

Next we construct $2n$ integers $v_i, v'_i$ for $i\in\{1,\ldots,n\}.$
$$ v_{ij}=\begin{cases} 1, & \mbox{if } j=i \mbox{ or } x_i\in C_{j-n}\\0, & \mbox{otherwise} \end{cases},$$
$$ {v'}_{ij}=\begin{cases} 1, & \mbox{if } j=i \mbox{ or } \overline{x}_i\in C_{j-n}\\0, & \mbox{otherwise} \end{cases}.$$

Finally, we construct $2m$ integers $s_i, s'_i$ for $i\in\{1,\ldots,m\}.$
$$ s_{ij}=\begin{cases} 1, & \mbox{if } j=n+i \\0, & \mbox{otherwise} \end{cases},$$
$$ s'_{ij}=\begin{cases} 2, & \mbox{if } j=n+i \\0, & \mbox{otherwise} \end{cases}.$$

We define the set $A(\Phi)=\{v_1,\ldots,v_n,v'_1,\ldots,v'_n,s_1,\ldots,s_{m},s'_1,\ldots,s'_m\}.$ Table \ref{tab:subset_sum} illustrates the construction for the 3SAT formula $(x_{1}\lor\overline{x}_{2}\lor x_{3})\land(\overline{x}_{1}\lor{x_{2}}\lor {x_{3}})\land({x_{1}}\lor{x_{2}}\lor \overline{x}_{3})\land(\overline{x}_{1}\lor\overline{x}_{2}\lor \overline{x}_{3}).$

\begin{table}[!ht]
 \centering
 \begin{tabular}{|r c c c c c c c c|}
  \hline
  & & $x_1$ & $x_2$ & $x_3$ & $C_1$ & $C_2$ & $C_3$ & $C_4$ \\
  \hline
  $v_1$ & = & 1 & 0 & 0 & 1 & 0 & 1 & 0 \\
  $v'_1$ & = & 1 & 0 & 0 & 0 & 1 & 0 & 1 \\
  $v_2$ & = & 0 & 1 & 0 & 0 & 1 & 1 & 0 \\
  $v'_2$ & = & 0 & 1 & 0 & 1 & 0 & 0 & 1 \\
  $v_3$ & = & 0 & 0 & 1 & 1 & 1 & 0 & 0 \\
  $v'_3$ & = & 0 & 0 & 1 & 0 & 0 & 1 & 1 \\
  $s_1$ & = & 0 & 0 & 0 & 1 & 0 & 0 & 0 \\
  $s'_1$ & = & 0 & 0 & 0 & 2 & 0 & 0 & 0 \\
  $s_2$ & = & 0 & 0 & 0 & 0 & 1 & 0 & 0 \\
  $s'_2$ & = & 0 & 0 & 0 & 0 & 2 & 0 & 0 \\
  $s_3$ & = & 0 & 0 & 0 & 0 & 0 & 1 & 0 \\
  $s'_3$ & = & 0 & 0 & 0 & 0 & 0 & 2 & 0 \\
  $s_4$ & = & 0 & 0 & 0 & 0 & 0 & 0 & 1 \\
  $s'_4$ & = & 0 & 0 & 0 & 0 & 0 & 0 & 2 \\
  \hline
  $b$ & = & 1 & 1 & 1 & 4 & 4 & 4 & 4 \\  
  \hline
 \end{tabular}
 \caption{The base $10$ numbers created as an instance of subset-sum for the 3SAT formula $(x_{1}\lor\overline{x}_{2}\lor x_{3})\land(\overline{x}_{1}\lor{x_{2}}\lor {x_{3}})\land({x_{1}}\lor{x_{2}}\lor \overline{x}_{3})\land(\overline{x}_{1}\lor\overline{x}_{2}\lor \overline{x}_{3}).$}\label{tab:subset_sum}
\end{table}

Consider the subset-sum instance with $A(\Phi), b$ as constructed above for any 3SAT instance $\Phi.$ Let $S$ be any subset of $A(\Phi).$ If the elements of $S$ sum exactly to $b$ then it is clear that for each $i\in\{1,\ldots,n\}$ exactly one of $v_i,v'_i$ belong to $S.$ Furthermore, setting $x_i=1$ if $v_i\in S$ or $x_i=0$ if $v'_i\in S$ satisfies every clause. Thus the characteristic vector of $S$ restricted to $\{v_1,\ldots,v_n\}$ is a satisfying assignment for the corresponding SAT formula.

Also, if $\Phi$ is satisfiable then the instance of subset sum thus created has a solution corresponding to each satisfying assignment: Pick $v_i$ if $x_i=1$ or $v'_i$ if $x_i=0$ in an assignment. Since the assignment is satisfying, every clause is satisfied and so the sum of digits corresponding to each clause is at least $1.$ Therefore, for a clause $C_j$ either $s_j$ or $s'_j$ or both can be picked to ensure that the sum of the corresponding digits is exactly $4.$ Note that there is unique way to do this.

This shows that every vertex of the subset sum polytope $\SUBSETSUM(A(\Phi),b)$ projects to a vertex of $\SAT(\Phi)$ and every vertex of $\SAT(\Phi)$ can be lifted to a vertex of $\SUBSETSUM(A(\Phi),b),$.
The projection is defined by dropping every coordinate except those corresponding to the numbers $v_i$ in the reduction described above. 
The lifting is defined by the procedure in the proceeding paragraph.
Hence, $\SAT(\Phi)$ is a projection of $\SUBSETSUM(A(\Phi),b).$
\end{proof}

Combining the preceding two theorems we obtain the following.
\begin{corollary}\label{cor:subsetsum_lb}For every natural number $n\geqslant 1,$ there exists an instance $A,b$ of the subset-sum problem with $O(n)$ integers in $A$ such that $\xc(\SUBSETSUM(A,b))\geqslant 2^{\Omega(\sqrt n)}.$
\end{corollary}

As mentioned above, the polytope $\SUBSETSUM(A,b)$ is a face of $\Knapsack(A,b)$ and hence Corollary \ref{cor:subsetsum_lb} implies a superpolynomial lower bound for the Knapsack polytope. We would like to note that a similar bound for the Knapsack polytope was proved recently and independently by Pokutta and van Vyve \cite{PV13}.

\subsection{3d-matching}
Consider a hypergraph $G=([n],E)$, where $E$ contains triples for some $i,j,k\in [n]$ where $i,j,k$ are distinct. A subset $E'\subseteq E$ is said to be a 3-dimensional matching if all the triples in $E'$ are disjoint. The $3d$-matching polytope $\threeDM(G)$ is defined as the convex hull of the characteristic vectors of every $3d$-matching of $G.$ That is,
$$ \threeDM(G):=\conv(\{\chi(E')\mid E'\subseteq E ~~~\text{is a $3d$-matching}\})$$

It is often customary to consider only hypergraphs defined over three disjoint set of vertices $X,Y,Z$ such that the hyperedges are subsets of $X\times Y\times Z.$ Observe that any hypergraph $G$ can be converted into a hypergraph $H$ in such a form by making three copies of the vertex set $V,V',V''$ and using a hyperedge $(i,j',k'')$ in $H$ if and only if $(i,j,k)$ is a hyperedge in $G.$ It is easy to see that $\xc(\threeDM(G))=\Theta(\xc(\threeDM(H))).$

\medskip
The 3d-matching problem asks: given a hypergraph $G$, does there exist a 3d-matching that covers all vertices? This problem is known to be $NP$-complete and was one of Karp's 21 problems proved to be $NP$-complete \cite{GJ,Karp}. Note that this problem can be solved by 
linear optimization over the polytope $\threeDM(G)$ and therefore it is to be expected that $\threeDM(G)$ would not have a polynomial size extended formulation.
 
In this subsection, we show that the 3d-matching polytope has superpolynomial extension complexity in the worst case. We prove this using a standard reduction from 3SAT to 3d-Matching used in the NP-completeness proof for the later problem (See \cite{GJ}). 
The form of this reduction, which is very widely used, employs a gadget for each variable along with a gadget for each clause. 
We omit the exact details for the reduction here because we are only interested in the correctness of the reduction and the variable gadget (See Figure \ref{fig:3d_matching_variable_gadget}).

\begin{figure}[!ht]
  \centering
    \includegraphics[width=0.5\textwidth]{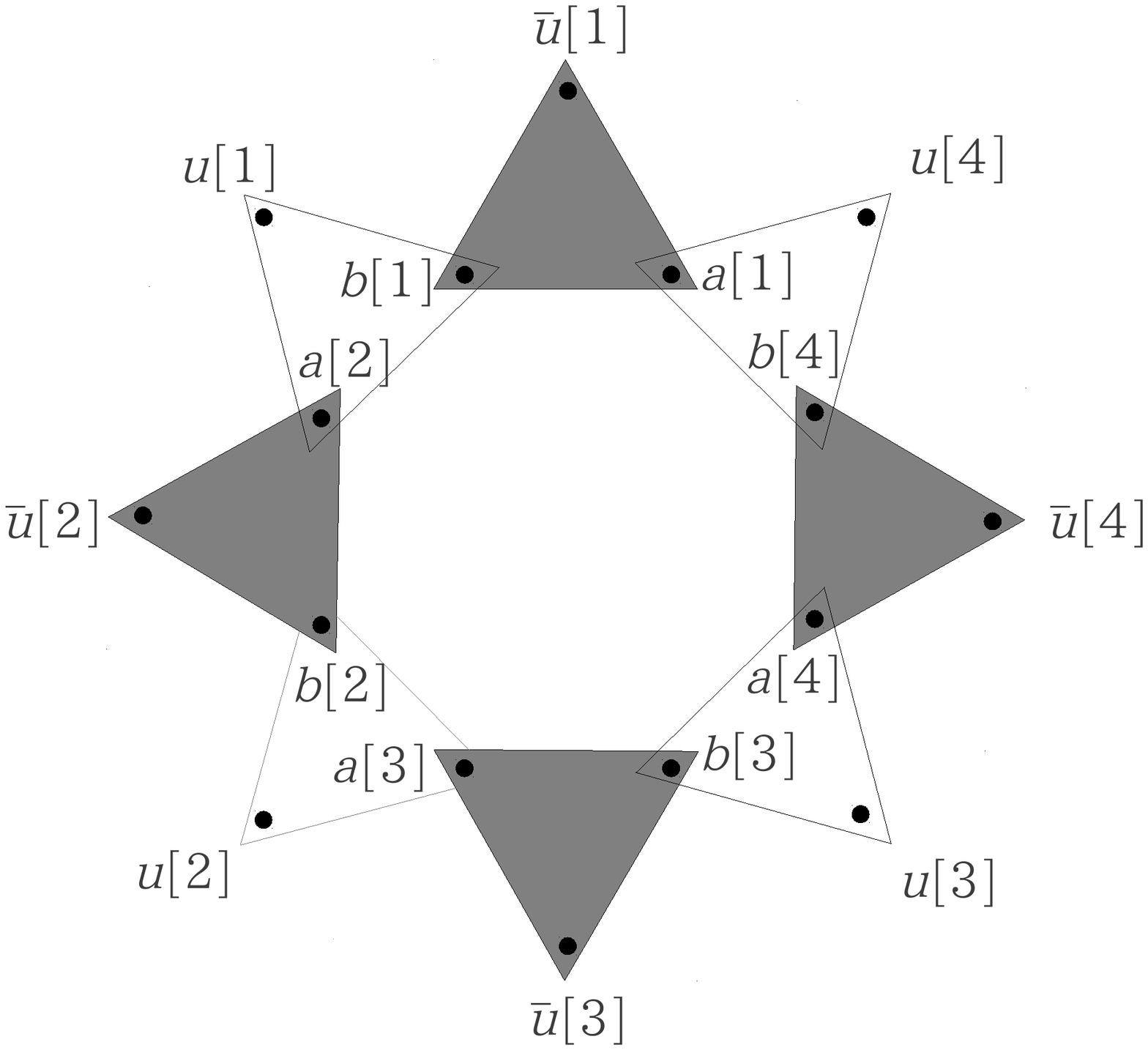}
    \caption{Gadget for a variable .}  
    \label{fig:3d_matching_variable_gadget}
\end{figure}

In the reduction, any 3SAT formula $\Phi$ is converted to an instance of a 3d-matching by creating a set of hyperedges for every variable (See Figure \ref{fig:3d_matching_variable_gadget}) along with some other hyperedges that does not concern us for our result. The crucial property that we require is the following: any satisfiable assignment of $\Phi$ defines some (possibly more than one) 3d-matching. Furthermore, in any maximal
matching either only the light hyperedges or only the dark hyperedges are picked, corresponding to setting the corresponding variable to, say, true or false respectively. Using these facts we can prove the following:

\begin{theorem}
Let $\Phi$ be an instance of 3SAT and let $H$ be the hypergraph obtained by the reduction above. Then $\SAT(\Phi)$ is the projection of a face of $\threeDM(H).$
\end{theorem}
\begin{proof}
Let the number of hyperedges in the gadget corresponding to a variable $x$ be $2k(x).$ Then, the number of hyperedges picked among these hyperedges in any matching in $H$ is at most $k(x).$ Therefore, if $y_1,\ldots,y_{2k(x)}$ denote the variables corresponding to these hyperedges in the polytope $\threeDM(H)$ then $\sum_{i=1}^{2k(x)}y_i \leqslant k(x)$ is a valid inequality for $\threeDM(H).$ Consider the face $F$ of $\threeDM(H)$ obtained by adding the equality $\sum_{i=1}^{2k(x)}y_i = k(x)$ corresponding to each variable $x$ appearing in $\Phi.$

Any vertex of $\threeDM(H)$ lying in $F$ selects either all light hyperedges or all dark hyperedges. Therefore, projecting out all variables except one variable $y_i$ corresponding to any fixed (arbitrarily chosen) light hyperedge for each variable in $\Phi$ gives a valid satisfying assignment for $\Phi$ and thus a vertex of $\SAT(\Phi).$ Alternatively, any vertex of $\SAT(\Phi)$ can be extended to a vertex of $\threeDM(H)$ lying in $F$ easily.

Therefore, $\SAT(\Phi)$ is the projection of $F.$
\end{proof}

The number of vertices in $H$ is $O(nm)$ where $n$ is the number of variables and $m$ the number of clauses in $\Phi.$ Considering only the 3SAT formulae with high extension complexity from subsection \ref{subsec:3sat}, we have $m=O(n).$ Therefore, considering only the hypergraphs 
arising from such 3SAT formulae and using propositions 1 and 2, we have that 

\begin{corollary}For every natural number $n\geqslant 1,$ there exists a hypergraph $H$ with $O(n)$ vertices such that $\xc(\threeDM(H))\geqslant 2^{\Omega(n^{1/4})}.$
\end{corollary}

%
%

\subsection{Stable set for cubic planar graphs}
Now we show that $\STAB(G)$ can have superpolynomial extension complexity even when $G$ is a cubic planar graph. Our starting point is the following result proved by Fiorini et. al. \cite{FMPTW}.

\begin{theorem}[{\cite{FMPTW}}]\label{thm:xc_stab}
For every natural number $n\geqslant 1$ there exists a graph $G$ such that $G$ has $O(n)$ vertices and $O(n)$ edges, and $\xc(\STAB(G))\geqslant 2^{\Omega(\sqrt n)}.$
\end{theorem}

We start with this graph and convert it into a cubic planar graph $G'$ with $O(n^2)$ vertices and extension complexity at least $2^{\Omega(\sqrt n)}.$

\subsubsection{Making a graph planar}
For making any graph $G$ planar without reducing the extension complexity of the associated stable set polytope, we use the same gadget used by Garey, Johnson and Stockmeyer \cite{GJS} in the proof of NP-completeness of finding maximum stable set in planar graph. Start with any planar drawing of $G$ and replace every crossing with the gadget $H$ with 22 vertices shown in Figure \ref{fig:planar_gadget} to obtain a graph $G'$. The following theorem shows that $\STAB(G)$ is the projection of a face of $\STAB(G').$

\begin{figure}[h!]
  \centering
    \includegraphics[width=0.5\textwidth]{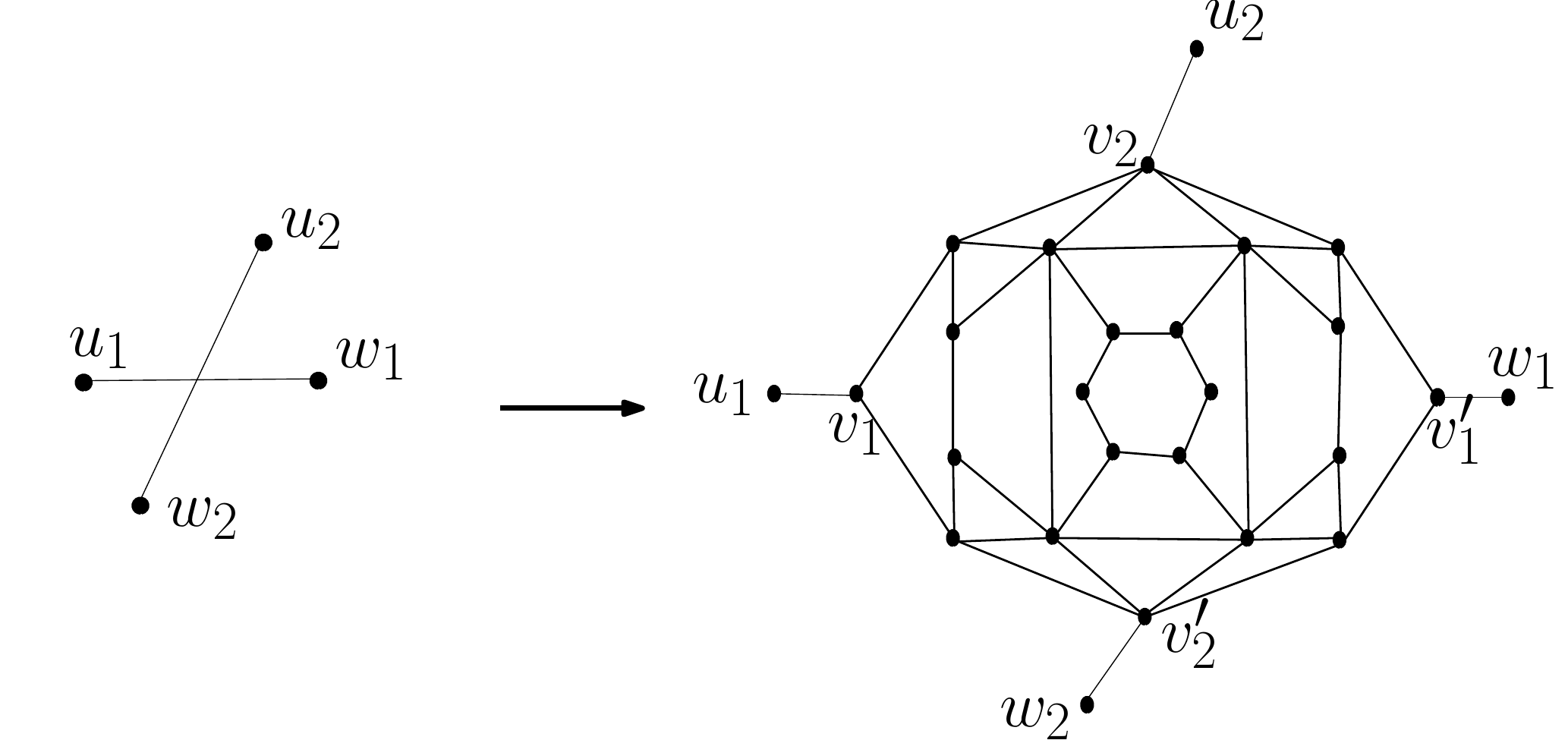}
    \caption{Gadget to remove a crossing.}  
    \label{fig:planar_gadget}
\end{figure}

\begin{theorem} \label{thm:making_a_graph_planar} Let $G$ be a graph and let $G'$ be obtained from a planar embedding of $G$ by replacing every edge intersection with a gadget shown in Figure \ref{fig:planar_gadget}. Then, $\STAB(G)$ is the projection of a face of $\STAB(G').$ 
\end{theorem}
\begin{proof}
Let $H_1,\ldots,H_k$ be the gadgets introduced in $G$ to obtain $G'.$ Any stable set $S$ of $G'$ contains some, or possibly no,
vertices from the gadgets introduced. For any gadget $H \in \{H_1,\ldots,H_k\},$ let $V_{H}$ denote the set of vertices of $H.$ Then, $S\cap V_{H}$ is a stable set for $H.$ Denote by $s_{ij}$ the size of maximum independent set in $H$ containing exactly $i$ vertices out of $\{v_1,v'_1\}$ and exactly $j$ vertices out of $\{v_2,v_2'\}.$ Table \ref{tab:gjs} lists the values of $s_{ij}$ for $i,j\in\{1,2\}.$ The table is essentially Table 1 from \cite{GJS} but their table lists the size of the minimum vertex cover and so we subtract the entries from the number of nodes in the gadget which is $22.$
\begin{table}[!ht]
\caption{Values of $s_{ij}$}\label{tab:gjs}
\begin{center}
\begin{tabular}{|c|c|c|c|}
\hline
$i\backslash j$ & $2$ & $1$ & $0$ \\ \hline \cline{2-4}
$2$ & $9$ & $8$& $7$ \\ \hline
$1$ & $9$ & $9$ & $8$ \\ \hline
$0$ & $8$ & $8$ & $7$ \\ \hline
\end{tabular}
\end{center}
\end{table}

As we see, every stable set of $H$ has fewer than $9$ vertices and hence $\sum_{i\in V_H} x_i \leqslant 9$ is a valid inequality for $\STAB(G').$ Consider the face $$F:= \STAB(G')\bigcap_{i=1}^k\{x\mid\sum_{j\in V_{H_i}}x_j=9\}$$

Consider any stable set $S$ of $G'$ lying in the face $F.$ It is clear that at least one vertex must be picked in $S$ out of each $\{v_1,v'_1\}$ and $\{v_2,v'_2\}.$ Therefore, for any edge $(u,v)$ in $G$ it is not possible that both $u,v$ are in $S$ and hence projecting out the vertices from the gadgets we get a valid stable set for $G.$ Alternatively, any independent set from $G_n$ can be extended to a stable set in $G'$ by selecting the appropriate maximum stable set from each of the gadgets. Therefore, $\STAB(G)$ is a projection of $F.$
\end{proof}

Since for any graph $G$ with $O(n)$ edges, the number of gadgets introduced $k \leqslant O(n^2),$  we have that the graph $G'$ in the above theorem has at most $O(n^2)$ vertices and edges. Therefore we have a planar graph $G'$ with at most $O(n^2)$ vertices and $O(n^2)$ edges.  This together with Theorem \ref{thm:xc_stab}, Theorem \ref{thm:making_a_graph_planar} and propositions 1 and 2 yields the following corollary.

\begin{corollary}\label{cor:stab_planar}
For every $n$ there exists a planar graph $G$ with $O(n^2)$ vertices and $O(n^2)$ edges such that $\xc(\STAB(G))\geqslant 2^{\Omega(\sqrt n)}.$
\end{corollary}

\subsubsection{Making a graph cubic} 
Suppose we have a graph $G$ and we transform it into another graph $G'$ by performing one of the following operations:
\begin{enumerate}
\item[]\textbf{ReduceDegree:} Replace a vertex $v$ of $G$ of degree $\delta\geqslant 4$ with a cycle $C_v=(v_1,v'_1,\ldots,v_{\delta},v'_\delta)$ of length $2\delta$ and connect the neighbours of $v$ to alternating vertices $(v_1,v_2,\ldots,v_\delta)$ of the cycle. (See Figure \ref{fig:reduce_degree}) 
\item[]\textbf{RemoveBridge:} Replace any degree two vertex $v$ in $G$ by a four cycle $v_1,v_2,v_3,v_4.$ Let $u$ and $w$ be the neighbours of $v$ in $G.$ Then, add the edges $(u,v_1)$ and $(v_3,w).$ Also add the edge $(v_2,v_4)$ in the graph. (See Figure \ref{fig:remove_bridge})
\item[]\textbf{RemoveTerminal:} Replace any vertex with degree either two or three with a triangle. In case of degree one, attach any one vertex of the triangle to the erstwhile neighbour. 
\end{enumerate}

\begin{figure}[h!]
  \centering
  \begin{subfigure}[b]{0.45\textwidth}
                \centering
                \includegraphics[width=\textwidth]{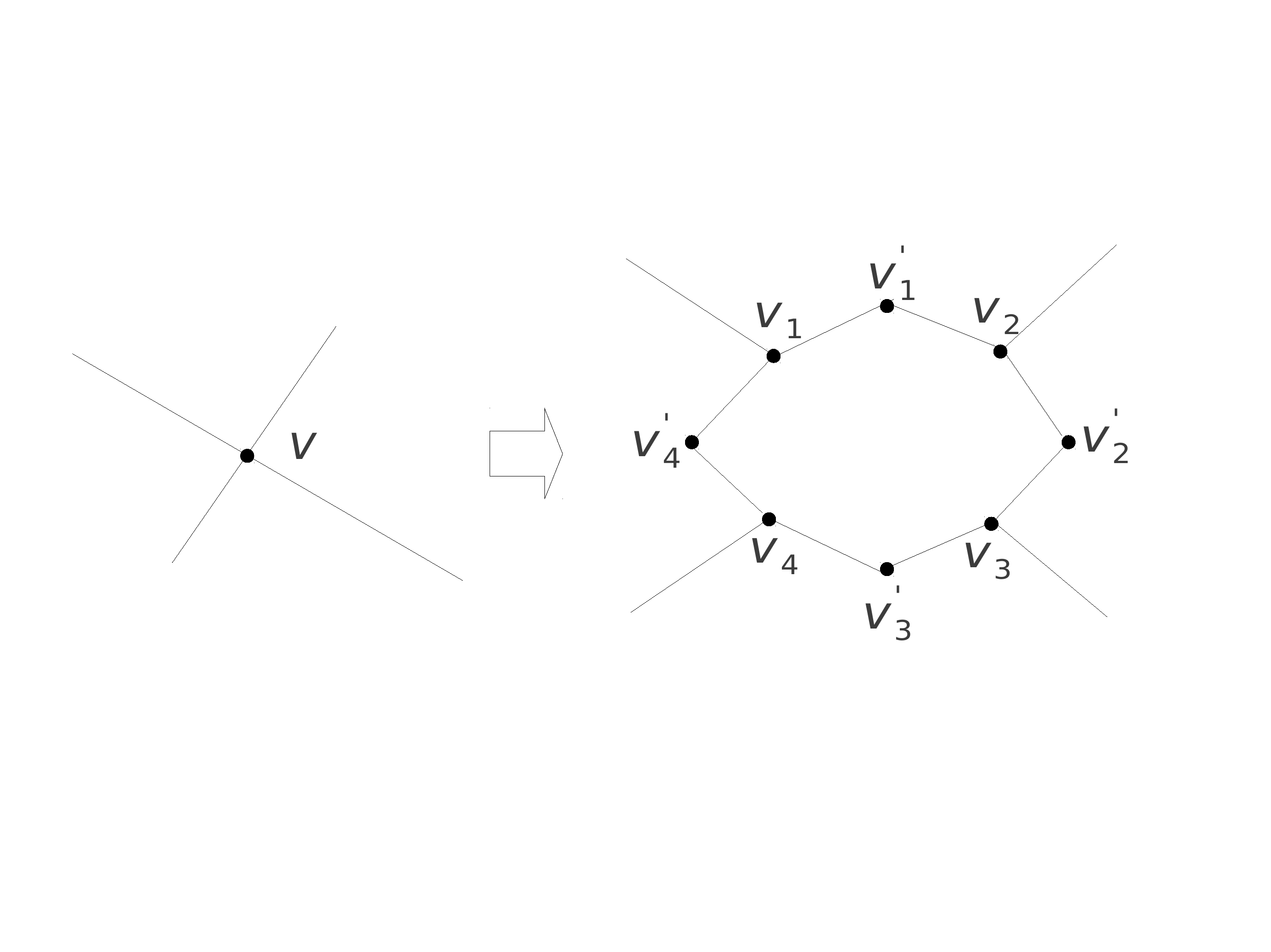}
                \caption{Replace a degree $4$ vertex.} 
                \label{fig:reduce_degree}
        \end{subfigure}
  \begin{subfigure}[b]{0.45\textwidth}
                \centering
                \includegraphics[width=\textwidth]{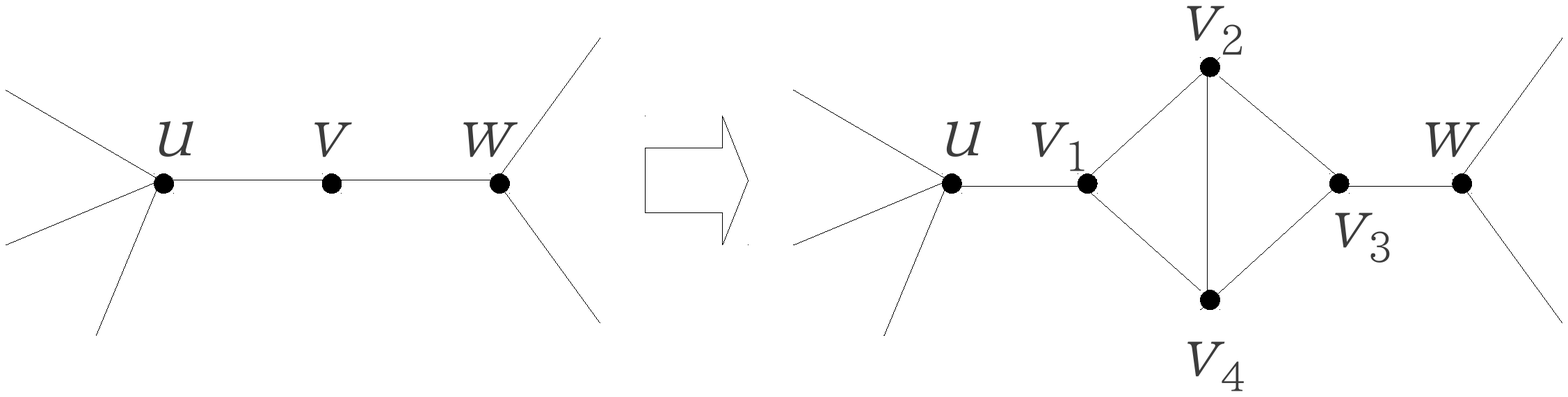}
                \caption{Remove a degree two vertex.}
                \label{fig:remove_bridge}
        \end{subfigure}
\caption{Gadgets}
\end{figure}

\begin{theorem} \label{thm:make_cubic}Let $G$ be any graph and let $G'$ be obtained by performing any number of operation ReduceDegree, RemoveBridge, or RemoveTerminal described above on $G$. Then $\STAB(G)$ is the projection of a face of $\STAB(G').$
\end{theorem}
\begin{proof}
It suffices to show that the theorem is true for a single application of either of the three operations.

Consider an application of the operation ReduceDegree. Let $C$ be the gadget that was used to replace a vertex $v$ in $G$ to obtain $G'.$ Let $V_{C}$ denote the set of vertices of $C.$ Then, for any stable set $S$ of $G',$ the set $S\cap V_{C}$ is a stable set for $C.$ Every stable set of $C$ has fewer than $\delta=|C|/2$ vertices and hence $\sum_{v\in V_C} x_v \leqslant |C|/2$ is a valid inequality for $\STAB(G').$ Consider the face $$F:= \STAB(G')\bigcap_{i=1}^k\{x\mid\sum_{v\in V_{C_i}}x_v=|C_i|/2\}$$

Any stable set $S$ lying in the face $F$ must either select all vertices $(v'_1,\ldots,v'_\delta)$ or $(v_1,\ldots,v_\delta)$ for each cycle $C$ of length $2\delta.$ Furthermore, if $S$ contains any neighbour of $v$ then the former set of vertices must be picked in $S.$ Also, any stable set of $G$ can be extended to a stable set of $G'$ that lies in $F.$ For each stable set in $F$ projecting out every vertex of the cycles introduced except any one that has degree $3$ gives us a valid stable set of $G$ and therefore, $\STAB(G)$ is the projection of a face of $\STAB(G').$

\medskip
 
On the other hand, suppose operation RemoveBridge is used to transform any graph $G$ into a graph $G'.$ Let $C$ be the gadget used to replace a vertex $v$ in $G.$ Let $V_{C}=(v_1,v_2,v_3,v_4)$ denote the set of vertices of $C.$ Then, for any stable set $S$ in $G',$ the set $S\cap V_{C}$ is a stable set for $C.$ It is easy to see that every stable set of $C$ satisfies the inequality $x_{v_1}+x_{v_3}+2(x_{v_2}+x_{v_4})\leqslant 2$ and hence it is a valid inequality for $\STAB(G').$ Define $h_C$ to be the equality obtained from the previous inequality for a gadget $C$ and consider the face $$F:= \STAB(G')\bigcap_{i=1}^k h_{C_i}$$

Any stable set $S$ of $G'$ lying in the face $F$ must either select vertices $(v_1, v_3)$ or one of $v_2$ or $v_4$ for each gadget $C.$ Furthermore, if $S$ contains any neighbour of $v$ then it contains exactly one of $v_2$ or $v_4$ but not both. Also, any stable set of $G$ can be extended to a stable set of $G'$ that lies in $F.$ For each stable set in $F$ projecting out every vertex of the gadget and using the map $x_v=x_{v_2}+x_{v_4}$ gives us a valid stable set of $G$ and therefore, $\STAB(G)$ is the projection of $F,$ a face of $\STAB(G').$ 

\medskip

Finally it is easy to see that if $G'$ is obtained by applying operation RemoveTerminal on a graph $G$ then $\STAB(G)$ is a projection of $\STAB(G').$
\end{proof}

If $G$ has $n$ vertices and $m$ edges then first applying operation ReduceDegree until every vertex has degree at most 3, and then applying operation RemoveBridge and RemoveTerminal repeatedly until no vertex of degree 0, 1 or 2 is left,  produces a graph that has $O(n+m)$ vertices and $O(n+m)$ edges. Furthermore, any application of the three operations do not make a planar graph non-planar. Combining this fact with Theorem \ref{thm:make_cubic}, Corollary \ref{cor:stab_planar} and propositions 1 and 2, we have

\begin{corollary}\label{cor:stab_cubic_planar}For every natural number $n\geqslant 1$ there exists a cubic planar graph $G$ with $O(n)$ vertices and edges such that $\xc(\STAB(G))\geqslant 2^{\Omega(n^{1/4})}.$
\end{corollary}

\section{Extended formulations for $\CutP(G)$ and its relatives}\label{sec:ext_cut}

We use the results described in the previous section to obtain bounds on
the extension complexity of the cut polytope of graphs.
We begin by reviewing the result in \cite{FMPTW} for $\CutP_n$
using a direct
argument that avoids introducing correlation polytopes.
For any integer $n \ge 2$ consider the integers $b_1=...=b_{n-1}=1$
and $b_{n}=3-n$. Let $b=(b_1, b_2, ... , b_n)$ be the corresponding $n$-vector.
Inequality (\ref{eq:hyp}) for this $b$-vector is easily
seen to be of negative type and can be written
\begin{equation} \label{eq:hyp2}
\sum_{1 \leq i < j \leq n-1} x_{ij} \leq 1 + (n-3) \sum_{i=1}^{n-1} x_{in}.
\end{equation}
\begin{lemma}\label{lem:slack}
Let $S$ be any cut in $K_n$ not containing vertex $n$ and let
$\delta(S)$ be its corresponding cut vector. 
Then the slack of $\delta(S)$ with respect to (\ref{eq:hyp2})
is ($|S|-1)^2$.
\end{lemma}
\begin{proof}
$$
1 + (n-3) \sum_{i=1}^{n-1} \delta(S)_{1i}-\sum_{1 \leq i < j \leq n-1} \delta(S)_{ij} 
= 1+(n-3)|S|-|S|(n-|S|-1)=|S|^2 -2|S|+1.
$$
\end{proof}
Let us label a cut $S$ by a binary $n$-vector $a$ where $a_i =1$ if and only
if $i \in S$. Under the conditions of the lemma we observe that the
slack $(|S|-1)^2= (a^Tb-1)^2$ since we have $a_n=0$ and $b_1=...=b_{n-1}=1$. 
Now consider
consider any subset $T$ of $\{1,2,...,n-1\} $ and set $b_i = 1$ for $i \in T$,
$b_n=3-|T|$ and $b_i = 0$ otherwise. We form a $2^{n-1}$ by $2^{n-1}$
matrix $M$ as follows. Let the rows and columns be indexed 
by subsets $T$ and $S$ of $\{1,2,...,n-1 \}$, labelled by the $n$-vectors
$a$ and $b$ as just described.
A straight forward application of Lemma \ref{lem:slack}
shows that $M = M^*(n-1)$. Hence using the fact that the non-negative rank of a matrix is at least as large as that of any of its submatrices, we have that every extended formulation of $\CutP_n$ has size $2^{\Omega(n)}$.

Recall the hypermetric polytope, defined in Section \ref{cutdef}, is the intersection of all hypermetric inequalities.
As remarked, nonnegative type inequalities are weaker than hypermetric inequalities and so valid for this polytope.
In addition all cut vertices satisfy all hypermetric inequalities. Therefore $M = M^*(n-1)$ is also a submatrix of
a slack matrix for the hypermetric polytope on $n$ points. So this polytope also has extension complexity at least $2^{\Omega(n)}$.

Finally let us consider the polytope, which we denote $P_n$, defined by the inequalities used to define rows of 
the slack matrix $M$ above.
We will show that membership testing for $P_n$ is co-NP-complete. 

\begin{theorem}
Let $P_n$ be the polytope defined as above, and let $x\in\RR^{{n(n-1)/2}}.$ Then it is co-NP-complete to decide if $x\in P_n.$
\end{theorem}
\begin{proof}
Clearly if $x \notin P_n$ then this can be witnessed by a violated inequality of type (\ref{eq:hyp2}),
so the problem is in co-NP.

To see the hardness we do a reduction from the clique problem: given graph $G=(V,E)$ on $n$ vertices and
integer $k$, does $G$ have a clique of size at least $k$? Since a graph has a clique of size $k$ if and only if its suspension has a clique of size $k+1$ we can assume \emph{wlog} that $G$ is a suspension with vertex $v_n$ connected to every other vertex.

Form a vector $x$  as follows:
$$ x_{ij}=
\begin{cases} 
1/k, & \mbox{if } j=n \\
2/k, & \mbox{if } j\neq n \mbox{ and } ij\in E\\
-n^2 & \mbox{otherwise}
\end{cases},$$


Fix an integer $t$, $2 \le t \le n$ and consider a $b$-vector with $b_n=3-t$, and with $t-1$ other values of $b_i=1$.
Without loss of generality we may 
assume these are lablelled $1,2,...,t-1$.
Let $T$ be the induced subgraph of $G$ on these vertices.
The corresponding non-negative type inequality is:
\begin{equation} \label{eq:hyp3}
\sum_{1 \leq i < j \leq t-1} x_{ij} \leq 1 + (t-3) \sum_{i=1}^{t-1} x_{in}.
\end{equation}

Suppose $T$ is a complete subgraph. Then the left hand side minus the right hand side of (\ref{eq:hyp3}) is
\[
\frac{2(t-1)(t-2)}{2k}-(1+\frac{(t-3)(t-1)}{k})=\frac{t-k-1}{k}.
\]
This will be positive if and only if $t \ge k+1$, in which case $x$ violates (\ref{eq:hyp3}).
On the other hand if $T$ is not a complete subgraph then the left hand side of (\ref{eq:hyp3})
is always negative and so the inequality is satisfied. Therefore  
$x$ satisfies all inequalities defining rows of $M$ if an only if $G$ has no clique of size at least $k$.
\end{proof}

\subsection{Cut polytope for minors of a graph}
A graph $H$ is a {\em minor} of a graph $G$ if $H$ can be obtained from $G$
by contracting some edges, deleting some edges and isolated vertices,
and relabeling. In the introduction we noted that if 
an $n$ vertex graph $G$ has no $K_5$-minor
then $\CutP(G)$ has $O(n^3)$ extension complexity. We will now show
that the extension complexity of a graph $G$ can be bounded from below
in terms of its largest clique minor. 

\begin{lemma}\label{lem:edge_deletion}
Let $G$ be a graph and let $H$ be obtained by deleting an edge $e$ of $G,$ then $\CutP(G)$ is an extension of $\CutP(H).$ In particular, $\xc(\CutP(G))\geqslant\xc(\CutP(H)).$
\end{lemma}
\begin{proof}
Any vertex $v$ of $\CutP(H)$ defines a cut on graph $H.$ Let $H_1$ and $H_2$ be the two subsets of vertices defined by this cut. Consider the same subsets over the graph $G,$ and the corresponding cut vector for $G$. This vector is the same as  $v$ extended with a coordinate corresponding to the edge $e$ in $G$ which was removed to obtain $H.$ The value on this coordinate is $0$ if the end points of this edge belong to different sides of the cut and $1$ otherwise. In either case, every vertex of $\CutP(G)$ projects to a vertex of $\CutP(H)$ and every vertex of $\CutP(H)$ can be lifted to a vertex of $\CutP(G).$

Therefore, $\CutP(G)$ is an extended formulation of $\CutP(H)$ and hence by Proposition \ref{prop:ef} $$\xc(\CutP(G))\geqslant\xc(\CutP(H)).$$
\end{proof}

\begin{lemma}\label{lem:vertex_deletion}
Let $G$ be a graph and let $H$ be obtained by deleting a vertex $v$ of $G,$ then $\CutP(G)$ is an extension of $\CutP(H).$ In particular, $\xc(\CutP(G))\geqslant\xc(\CutP(H)).$
\end{lemma}

The proof is analogous to that of Lemma \ref{lem:edge_deletion}.

\begin{lemma}\label{lem:edge_contraction}
Let $G$ be a graph and let $H$ be obtained by contracting an edge $e=(u,v)$ of $G,$ then $\CutP(H)$ is the projection of a face of $\CutP(G).$ In particular, $\xc(\CutP(G))\geqslant\xc(\CutP(H)).$
\end{lemma}
\begin{proof}
Suppose that the vertices $u,v$ are contracted to a new vertex labelled $u$ in $H.$ Consider the inequality $x_e\geqslant 0.$ This is a valid inequality for $\CutP(G).$ Consider the face $$F=\CutP(G)\cap\{x_e=0\}.$$

Consider any vertex of $F.$ Project out $x_e$ and also $x_{e'}$ for any $e'=(v,w)$ if $(u,w)$ is an edge in $G$. 
Clearly this linear map projects every vertex in $F$ to a vertex of $\CutP(H).$ 
Also, any vertex of $\CutP(H)$ can be lifted to a vertex of $\CutP(G)$ lying in $F$ as follows. 
Set $x_e=0,$ and for an edge $e'=(v,w)$ in $G$ we set $x_{vw}=x_{uw}.$ 
It is easy to check that this is a valid cut for $G$ that lies in $F.$ 

It is thus clear that $\CutP(H)$ is obtained as the projection of a face of $\CutP(G)$ by setting $x_e=0$ for the contracted edge $e.$ 
Hence by Proposition \ref{prop:face} $$\xc(\CutP(G))\geqslant\xc(\CutP(H)).$$
\end{proof}

Combining Lemma \ref{lem:edge_deletion}, \ref{lem:vertex_deletion}, and \ref{lem:edge_contraction} we get the following theorem.

\begin{theorem}
Let $G$ be a graph and $H$ be a minor of $G.$ Then, $$\xc(\CutP(G))\geqslant \xc(\CutP(H)).$$
\end{theorem}

Using the above theorem together with the result of \cite{FMPTW} that the extension complexity of $\CutP(K_n)$ is at least $2^{\Omega(n)}$ we get the following result.

\begin{corollary}
The extension complexity of $\CutP (G)$ for a graph $G$ with a $K_n$ minor
is at least $2^{\Omega(n)}$.
\end{corollary}

Using this theorem we can immediately prove that the Bell inequality polytopes
mentioned in the introduction
have exponential complexity.
\begin{corollary}
The extension complexity of
$\CutP (K_{1,n,n})$ is at least $2^{\Omega(n)}$.
\end{corollary}
\begin{proof}
Pick any matching of size $n$ between the vertices in each of the two parts of cardinality
$n$. Contracting the edges in this matching yields $K_{n+1}$ and the result follows
\end{proof}

\subsection{Cut Polytope for $K_6$ minor-free graphs}
Let $G=(V,E)$ be any graph with $V=\{1,\ldots,n\}.$ 
Consider the suspension $G'$ of $G$ obtained by adding an extra vertex labelled $0$ with edges to all vertices $V$. 

\begin{theorem}\label{thm:cut_suspension} Let $G=(V,E)$ be a graph and let $G'$ be a suspension over $G.$ Then $\STAB(G)$ is the projection of a face of $\CutP(G').$
\end{theorem}
\begin{proof}
The polytope $\STAB(G)$ is defined over variables $x_i$ corresponding to each of the vertex $i\in V$ whereas the polytope $\CutP(G')$ is defined over the variables $x_{ij}$ for $i,j\in\{0,\ldots,n\}.$

Any cut vertex $C$ of $\CutP(G')$ defines sets $S,\overline{S}$ such that $x_{ij}=1$ if and only if $i\in S, j\in \overline{S}.$ 
We may assume that $0\in\overline{S}$ by interchanging $S$ and $\overline{S}$ if necessary. 
For every edge $e=(k,l)$ in $G$ consider an inequality $h_e:= \{x_{0k}+x_{0l}-x_{kl}\geqslant 0\}.$ It is clear that $h_e$ is a valid inequality for $\CutP(G')$ for all edges $e$ in $G.$ Furthermore, $h_e$ is tight for a cut vector in $G'$ if and only if either $k,l$ do not lie in the same cut set or $k,l$ both lie in the cut set containing $0$. Therefore consider the face $$F:= \CutP(G')\bigcap_{(i,j)\in E}\{x_{0i}+x_{0j}-x_{ij}= 0\}.$$

Each vertex in $F$ can be projected to a valid stable set in $G$ by projecting onto the variables $x_{01},x_{02},\ldots,x_{0n}.$ Furthermore, every stable set $S$ in $G$ can be extended to a cut vector for $G'$ by taking the cut vector corresponding to $S,\overline{S}\cup\{0\}.$ Therefore, $\STAB(G)$ is the projection of a face of $\CutP(G').$
\end{proof}

Using this theorem it is easy to show the existence of graphs with a linear number of edges
that do not have $K_6$ as a minor and yet have a high extension complexity.
In fact we get a slightly sharper result.

\begin{theorem}
For every  $n\geqslant 2$ there exists a graph $G$ 
which is a suspension of a planar graph and for which $\xc(\CutP(G))\geqslant 2^{\Omega(n^{1/4})}.$
\end{theorem}
\begin{proof}
Consider a planar graph $G=(V,E)$ with $n$ vertices for which $\xc(\STAB(G))\geqslant 2^{\Omega(n^{1/4})}.$ 
Corollary \ref{cor:stab_planar} guarantees the existence of such a graph for every $n$. 
Then the suspension over $G$ has $n+1$ vertices and a linear number of edges.
The theorem then follows by applying Theorem \ref{thm:cut_suspension} together with Propositions 1 and 2.
\end{proof}

The above theorem provides a sharp contrast for the complexity of the cut polytope for graphs in terms of their minors. As noted in the introduction, for any $K_5$ minor-free graph $G$ with $n$ vertices $\CutP(G)$ has an extension of size $O(n^3)$ whereas the above result shows that there are $K_6$ minor free graphs whose cut polytope has superpolynomial extension complexity.
\section{Concluding remarks}
  \label{sect:conclusion}
We have a given a simple polyhedral procedure for proving lower bounds on the extension complexity of a polytope.
Using this procedure and some standard NP-completeness reductions we were able to prove lower bounds on the extension
complexity of various well known combinatorial polytopes. For the cut polytope in particular, we are able to draw a sharp
line, in terms of minors, for when this complexity becomes super polynomial.

Nevertheless the procedure is not completely `automatic' in the sense that any NP-completeness reduction of a certain type,
say using gadgets, automatically gives a result on the extension complexity of related polytopes.
This would seem to be a very promising line of future research.

\section*{Acknowledgments}

Research of the first author is supported by the NSERC and JSPS. Research of the second author is supported by FNRS.

\bibliographystyle{plain}
\bibliography{extform}

%
%
%
%
%
%
%
%
%

\end{document}